\documentclass[a4paper,12pt]{amsart}

\usepackage{amsmath,amsfonts,amssymb,amsthm}
\usepackage{url}

\numberwithin{equation}{section}

\newtheorem{Proposition}{Proposition}[section]
\newtheorem{Theorem}[Proposition]{Theorem}

\newtheorem{Definition}[Proposition]{Definition}
\newtheorem{Remark}[Proposition]{Remark}

\newcommand{\R}{\mathbb{R}}

\usepackage{graphicx}
\graphicspath{{./pictures/}}

\begin{document}

\title[Global isochrony]{Global isochronous potentials}

\date{Feb. 15, 2013. To appear in: Qualitative Theory of Dynamical Systems}

\begin{abstract} 
We present  a geometric characterization of the nonlinear smooth functions $V:\R\to \R$ for which the origin
is  a global isochronous center for the scalar equation $\ddot x=-V'(x)$. We revisit  Stillinger and Dorignac isochronous potentials $V$ and show a new simple explicit family. Implicit examples are easily produced.
\end{abstract}

\maketitle
\begin{center}
\emph{Dedicated to Jorge Lewowicz for his $75^{th}$ birthday}
\end{center}

\section{Introduction}  

Consider the scalar differential equation $\ddot x=-g(x)$ under the minimal assumptions that $g$ is continuous, $g'(0)$ exists and it is~$>0$, and that $xg(x)>0$ whenever $x\ne0$. Let $V$ be a primitive of~$g$, with $V(0)=0$. As it is well-known, the Cauchy problems have unique solution, thanks to the energy first integral $\dot x^2/2+V(x)$, all solutions with small enough energy are periodic, and their orbits in the $x,\dot x$ plane encircle the origin $(0,0)$, which is called a \emph{center} for the system. If moreover $V(x)\to+\infty$ as $x\to\pm\infty$, then all solutions are periodic, and the origin $(0,0)$ is a \emph{global center}. When all orbits around a (global) center have the same period the (global) center is called \emph{isochronous}.

After the 1937 paper~\cite{KP} by Koukles and Piskounov, a constructive characterizations of local isochronous centers was proposed by Urabe~\cite{U,U2} in 1961-62, in terms of a differential equation that must be satisfied by~$V$, and that contains an arbitrary odd function. Urabe's method has been cited and developed by many authors, for example recently by  Chouikha, see Strelcyn~\cite{SJ} and the reference therein. 

In 1989 the paper~\cite{z3} gave  a completely different characterization of local isochronous potentials (Theorem~\ref{isochronous} below). It is more constructive than Urabe's method, because we can start with any \emph{involution}~$h$ (Definition~\ref{involution} below), plug it into the simple formula $c(x-h(x))^2$ and the result is an isochronous potential. The involution condition has a very clear geometric meaning: the graph of~$h$ must be symmetric with respect to the diagonal. The involution criterion was independently rediscovered by Cima, Ma\~nosas and Villadelprat~\cite{cima1} in~1999 among several other results on
more general Hamiltonian systems.

In the last few years the literature has seen some resurgent interest in general
isochronous systems (Calogero's book~\cite{Ca} is a prime example),
  and particularly in characterizing \emph{global isochronous potentials} and in exhibiting elementary \emph{explicit examples}, beside the trivial linear oscillator. To our knowledge, the earliest global analytic example was given by Stillinger and Stillinger~\cite{S} in 1989 in a journal addressed to the chemical community. A~subclass of Stillingers' example was rediscovered by Bolotin and MacKay~\cite{BM} in 2003, together with necessary and sufficient conditions for a global center. Dorignac~\cite{dorignac} in 2005 was the first to cite the Stillingers, produced new examples and noticed that we can introduce a rescaling parameter. Among the most recent papers we may cite the ones by Asorey et al.~\cite{Asorey2007} and by Mamode~\cite{mamode}. 

Two papers that concern us most closely appeared in 2011. The first~\cite{z4}  gave a hopefully more accessible account of the involution method, and then it connected it to    some
Hamiltonian systems in dimension 4 which are interesting for Lyapunov stability, also finding
local isochrony examples by means of that dynamics  (see also Section~\ref{stillinger} below). The second paper~\cite{cima3}, by Cima, Gasull and Ma\~nosas, included the construction of an explicit global family of nonlinear centers, which was shown to be isochronous by the technique of involutions too. It was this paper which first drew our attention
to the problem of global examples.

The purpose of this note is to further demonstrate how the involution method is very well suited for the task of finding plenty of explicit global isochronous centers. The original paper~\cite{z3} suggested that involutions can be found by rotating the graph of an even function, but it kept the discussion local and did not pursue examples. Here we describe another \emph{geometric} idea which seems preferable for \emph{global} purposes: we look for involutions as (subsets of) the zero set of two-variable functions $f(x,y)$ which are symmetric, i.e., $f(x,y)=f(y,x)$ (Theorem~\ref{globalinvolution}). Two new examples will be readily obtained this way in Section~\ref{centriglobali}. 

The idea of using symmetric $f(x,y)$ to obtain involutions is elementary, and we are thankful to Armengol Gasull for pointing out to us that something  similar, although with applications unrelated to ours, was described in the 2002 paper~\cite{WW}, which in turn draws from the 1967~\cite{SM}.

In Section~\ref{stillinger} we show that Stillingers' examples can be obtained by choosing $f$ to be a symmetric quadratic polynomial, whose zero set is a hyperbola. It also shows that an example from~\cite{z4} contains Stillingers'. Section~\ref{dorignac} finally finds a symmetric function~$f$ that yields one   of Dorignac's examples
that we revisit because of its remarkable simplicity.

\section{Constructing isochronous centers}\label{constructing}

The concept of involution is crucial in our approach. Basically, a function $h\colon A\to A$ is an involution if $h\circ h$ is the identity mapping on~$A$. In this paper we will use the term ``involution'' to mean something more specific, as follows:

\begin{Definition}\label{involution} A $C^1$ diffeomorphism $h$ of an  open interval $J\subseteq \R$ onto itself is called an involution if
\begin{equation}\label{h-properties}
  h^{-1}=h,\qquad 0\in J,\qquad
  h(0)=0,\qquad h'(0)=-1.
\end{equation}
\end{Definition}

Geometrically speaking, the condition $h^{-1}=h$ means that the graph of~$h$ is symmetric with respect to the diagonal; indeed  $(x, h(x))$ has $(h(x),x)$ as symmetric point and this coincides
with the point $(h(x),\allowbreak h(h(x)))$ of the graph. The condition $h(0)=0$ means that the graph intersects the diagonal at the origin. For example, start with the hyperbola $y\,x=1$, which is symmetric with respect to the diagonal. If we translate its point $(1,1)$ to the origin we get $(y+1)(x+1)=1$, which can be solved for $y$ as $y=-x/(1+x)$. If we finally take the branch that goes through the origin we arrive at the following involution
\begin{equation}\label{h-upperrectangularhyperbola}
  h(x)=-\frac{x}{1+x},\qquad x\in J=(-1,+\infty)\,.
\end{equation}

A homothety of the plane transforms an involution into another involution:

\begin{Remark}\label{remarkfamily} Let $a\in\R\setminus\{0\}$ and $h$ be an involution on $(b,c)$, then
 $\tilde h(x)=h(a\,x)/a$ is an involution on $(b/a,c/a)$ if $a>0$, on $(c/a,b/a)$ otherwise.
\end{Remark}
In particular $\tilde h^{-1}=\tilde h$, indeed
$\tilde h\bigl(\tilde h(x)\bigr)=h\bigl(a\,h(a\,x)/a\bigr)/a=x$.
In this way \eqref{h-upperrectangularhyperbola} gives the following 1-parameter family of involutions
\begin{equation}\label{h-rational}
  h(x)=-\frac{x}{1+a x}\,, \qquad 
  x\in J=\begin{cases}
  (-1/a,+\infty),&a>0\\
  (-\infty,+\infty),&  a=0\\
  (-\infty,-1/a),& a<0\end{cases}
\end{equation}
These are the only involutions that are rational functions of~$x$, as proved in~\cite{cima1}.

\begin{Theorem}\label{isochronous}
Let $h:J\to J$ be an involution, $\omega>0$, and define
\begin{equation}\label{Visochronous}
  V(x)=\frac{\omega^2}{8}\,\big(x-h(x)\big)^2,\qquad
  x\in J.
\end{equation}
Then  the origin is an isochronous center for $\ddot{x}=-g(x)$, where $g(x)=V'(x)$, namely all orbits which intersect the $J$ interval of the $x$-axis in the $x,\dot x$-plane, are periodic and have the same period $2 \pi/\omega$. Vice versa, let $g$ be continuous on a neighbourhood of $0\in\R$, $g(0)=0$, suppose there exists $g'(0)>0$, and the origin is an isochronous center for $\ddot x=-g(x)$,  then there exist an open interval $J$, $0\in J$, which is a subset of the domain of $g$, and an involution $h:J\to J$  such that \eqref{Visochronous} holds with $V(x)=\int_0^xg(s)ds$ and $\omega= \sqrt{g'(0)}$.
\end{Theorem}

The proof is included in the proof of Proposition~1 in~\cite{z3}, by one of the present authors,  as a particular case. Formula \eqref{Visochronous} corresponds to formula~(6.2) in the paper~\cite{z3}. A detailed proof can be also found in the recent~\cite{z4}; see Theorem~2.1 and Corollary~2.2 in~\cite{z4}. This last paper also contains the following local necessary conditions for isochrony:
\begin{equation}
  V^{(4)}(0)=\frac{5V^{3}(0)^2}{3V''(0)},\quad
  V^{(6)}(0)=\frac{7V'''(0)V^{(5)}(0)}{V''(0)}-
  \frac{140V'''(0)^4}{9V''(0)^3},
\end{equation}
which can be deduced by taking successive derivatives of the involution relation $h(h(x))\equiv x$ at $x=0$.

Inserting the involution~\eqref{h-rational} into formula~\eqref{Visochronous} we obtain the following isochronous potential
\begin{equation}\label{V-upperrectangularhyperbola}
  V(x)=\frac{\omega^2}{8}\,x^2
  \Bigl(\frac{2+a x}{1+a x}\Bigr)^2,\quad
  x\in J=\begin{cases} (-1/a,+\infty),&a>0\\
  (-\infty,+\infty),&a=0\\
  (-\infty,-1/a),&a<0.\end{cases}
\end{equation}
This is the only rational potential for which there is isochrony, as proved in~\cite{ChalykhVeselov} and in~\cite{Asorey2007}. Of course, it is not defined globally on~$\R$, except in the trivial case $a=0$.

Formula~\eqref{Visochronous} implies the following \emph{global} inequality:
\begin{equation}\label{globalInequality}
  V(x)\ge\frac{\omega^2}{8}x^2,
  \qquad x\in J.
\end{equation}
Simply observe that $x$ and $h(x)$ have opposite signs, so that $x-h(x)\ge x$ when $x\ge0$, and $x-h(x)\le x$ when $x\le0$. The inequality~\eqref{globalInequality} is meaningful not for small~$x$ (because $V(x)\sim \omega^2 x^2/2$ as $x\to0$) but rather as~$x\to \pm\infty$. In particular, global isochronous potentials are \emph{at least quadratic} at infinity. If $V(x)$ grows more than quadratically as~$x\to+\infty$, it must grow not more than quadratically as~$x\to-\infty$, and the other way round.

\section{Global centers}\label{centriglobali}

Our approach to centers on the whole $\R^2$ for  $\ddot x=-g(x)$ is based on the following geometric characterization of global involutions in terms of zero sets of symmetric functions of two variables:

\begin{Theorem}\label{globalinvolution}
Let $f:\Omega\to\R$ be a $C^1$ function on the open set $\Omega \subseteq\R^2$ such that: $(0,0)\in\Omega$, $f(0,0)=0$, and
\begin{equation}\label{symmetry}
  (x,y)\in \Omega\quad\Longrightarrow\quad
  (y,x)\in\Omega,\quad f(y,x)=f(x,y).
\end{equation}
Let $\Gamma$ be the connected component of $f^{-1}(0)$ that contains the origin. Suppose that $\partial_2f(x,y)\ne 0$ for all $(x,y)\in\Gamma$, and that $\Gamma$ project itself onto the whole $x$-axis, namely, for each $x\in\R$ we have $(x,y)\in\Gamma$ for some $y\in\R$. Then $\Gamma$ is the graph of a global involution $h$ that is defined on the whole~$\R$. All global involutions can be obtained this way.
\end{Theorem}

\begin{proof} By the implicit function theorem, $\Gamma$~is the graph of a $C^1$ function $h$ with $h(0)=0$. Moreover, $h$~is defined on the whole~$\R$ since $\Gamma$ projects itself onto the whole $x$-axis. From~\eqref{symmetry}
we have $\partial_1f(x,y)=\partial_2f(y,x)$ for $(x,y)\in\Omega$ so $\partial_1f$ never vanishes on $\Gamma$ and has the same sign as $\partial_2f$. We deduce that
\begin{equation*}
  h'(x)=-\frac{\partial_1f\bigl(x,h(x)\bigr)}
  {\partial_2f\bigl(x,h(x)\bigr)}<0,\qquad x\in\R,
\end{equation*}
and in particular $h'(0)=-1$. Finally, $f(h(y),y)=f(y,h(y))=0$ shows that $h^{-1}=h$ and $h$ is a global involution.

Conversely, let $h:\R\to\R$ be a global involution and define $f(x,y):= x+y-h(x)-h(y)$. This is a $C^1$ function on $\R^2$, $f(0,0)=0$, and $f(y,x)=f(x,y)$. We have $\partial_2f(x,y)=1-h'(y)>0$ for all $(x,y)\in\R^2$. The graph of $h$ is a subset of the set $\Gamma$ in the statement of the theorem, so $\Gamma$ projects itself onto the $x$-axis (and it actually coincides with the graph of~$h$).
\end{proof}

\begin{figure}
\begin{center}
\includegraphics[width=.5\textwidth]{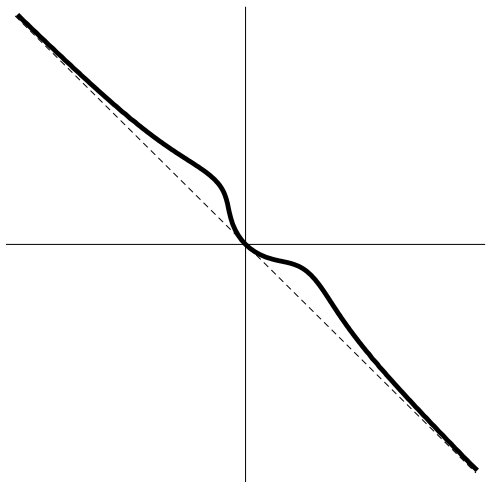}
\end{center}
\caption{Global involution defined by \eqref{example1}.}
\label{fighgi}
\end{figure}

The theorem can be applied to the $f$ of the form $f(x,y)= \varphi(x) +\varphi(y)$, where $\varphi$ is a diffeomorphism of~$\R$ onto itself, and $\varphi(0)=0$. For a non trivial example let us consider the following function of two variables
\begin{equation}\label{example1}
  f(x,y):=x-x^2+x^5+y-y^2+y^5,\quad (x,y)\in\R^2.
\end{equation}
All conditions in Theorem~\ref{globalinvolution} are satisfied since $\partial_2f(x,y)=1-2y+5y^4\ne 0$ for all $y\in\R$ and the function $x\mapsto x-x^2+x^5$ is a diffeomorphism $\R\to\R$ so the zero set of~$f$ projects itself onto the whole~$\R$. So a global (implicit) involution $h$ is defined. It is a non elementary algebraic function. In this example the zero set is connected and it coincides with $\Gamma$. The graph of~$h$ is depicted in Figure~\ref{fighgi}.

To get an example where there are more connected components, and $\partial_2f(x,y)=0$ at some $(x,y)\notin\Gamma$, we can simply modify $f$ as 
\begin{equation*}
  f(x,y)=\bigl(x-x^2+x^5+y-y^2+y^5\bigr)
  \bigl((x-1)^2+(y-1)^2-1\bigr),
\end{equation*}
which yields the same global involution (notice that $\partial_2 f(0,1)=0$).

Sometimes the equation $f(x,y)=0$ can even be solved explicitly. For instance for the family 
\begin{equation}\label{example2}
  f(x,y):=x+y+\rho\bigl(e^x+e^y-2\bigr)=0,\quad
  (x,y)\in\R^2, \rho\ge 0.
\end{equation}
Again, all conditions in Theorem~\ref{globalinvolution} are satisfied since $\partial_2f(x,y)=1+\rho e^y$ for all $y\in\R$ and the function $x\mapsto x+\rho e^x$ is a diffeomorphism $\R\to\R$ so the zero set of $f$ projects itself on the whole $\R$. In this case the global involution can be written in terms of  the Lambert  real function $W_0: [-1/e,+\infty)\to [-1,+\infty)$ (the principal branch of the
product logarithm so $W_0\bigl(xe^x)=x$), as
\begin{equation*}
  x\mapsto -x+\rho(2-e^x)-
  W_0\bigl(\rho\, e^{-x+\rho(2-e^x)}\bigr).
\end{equation*}
By Remark~\ref{remarkfamily} we finally obtain the two real parameters global family
\begin{multline}\label{h-Lambert}
  h(x)=-x+\frac{\rho}{a}(2-e^{ax})-
  \frac{1}{a}\,W_0\bigl(\rho\, e^{-{ax}+\rho(2-e^{ax})}\bigr),\\
  x\in\R,\rho\ge 0, a\ne 0.
\end{multline}
In Figure~\ref{fighl} we can see the graph of the involution $h$ for $a>0$. The growth of~$h(x)$ is exponential as $x\to+\infty$ and logarithmic as~$x\to-\infty$. Hence the corresponding isochronous potentials $V(x)$ will be exponential as $x\to+\infty$ and quadratic as~$x\to-\infty$.

\begin{figure}
\begin{center}
\includegraphics[width=.5\textwidth]{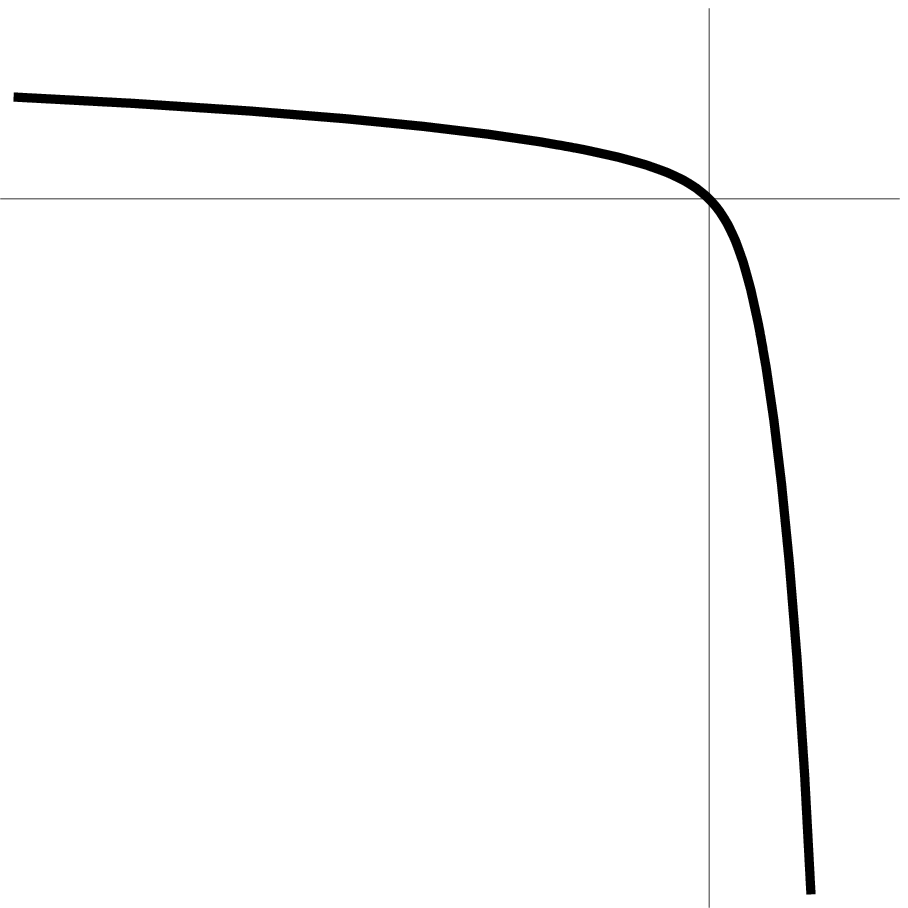}
\end{center}
\caption{Global involution \eqref{h-Lambert}.}
\label{fighl}
\end{figure}

\section{Stillinger centers from hyperbolas}\label{stillinger}

Let us consider the following symmetric quadratic function with two real parameters
\begin{multline}
  f(x,y)=\lambda (x^2+y^2)+2(\lambda+2a^2)xy+
  4a(\lambda+a^2)(x+y),\\
  (x,y)\in\R^2,\quad \lambda>0,\quad a\ne 0.
\end{multline}
The equation $f(x,y)=0$ gives a hyperbola with center $(-a,-a)$, vertices $(0,0),(-2a,-2a)$, and asymptotes 
\begin{equation}\label{asymptoteshyperbola}
  \lambda (y+a)=\bigl(-(\lambda+2a^2)\pm 
  2a\sqrt{\lambda+a^2}\bigr)(x+a).
\end{equation}
Both values $-(\lambda+2a^2)\pm 2a\sqrt{\lambda+a^2}<0$, so each branch of the hyperbola projects itself on the whole $x$-axis.
If $a>0$ the upper branch intersects the origin,  the lower one if $a<0$ (in both cases  for $x\to-\infty$ the asymptote of the branch through the origin is the one with the $+$ sign in \eqref{asymptoteshyperbola} and for $x\to+\infty$ the one with the $-$ sign).

All conditions in Theorem~\ref{globalinvolution} hold, in particular $f(0,0)=0$, $f(x,y)\allowbreak=f(y,x)$ for all $(x,y)\in\R^2$, and we can check that $\partial_2f(x,y)\ne 0$ for all points $(x,y)$ such that $f(x,y)=0$. By Theorem~\ref{globalinvolution}, the equation $f(x,y)=0$  defines a global involution $h:\R\to\R$.  In the present case we can even write down explicitly 
\begin{multline}\label{hStillinger}
  h(x):=-a-\frac{1}{\lambda}\Bigl((\lambda+2a^2)(a+x)+\\
  -2a\sqrt{\lambda+a^2}\,\sqrt{\lambda+(a+x)^2}\Bigr),\qquad 
  x\in\R,\lambda>0,
\end{multline}
where $a\in\R$, since $a=0$ gives the trivial $h(x)=-x$. 

Replacing the involution into formula~\eqref{Visochronous} we obtain the following isochronous potential
\begin{multline}\label{potentialStillingerinZampForm}
  V(x)=\omega^2 \,
  \frac{\lambda+a^2}{2\lambda^2}
  \biggl(\lambda a^2+(a+x)\Bigl((\lambda+2a^2)(a+x)+{}\\
  -2a\sqrt{\lambda+a^2}\;
  \sqrt{\lambda+(a+x)^2}\Bigr)\biggr),\qquad
  x\in\R,  \omega\ne 0, \lambda>0.
\end{multline}
So the origin is a global center for $\ddot x=-g(x)$ with $g(x)= V'(x)$, i.e.,
\begin{equation}\label{gStillingerinZampForm}
  g(x)=\omega^2
  \,\frac{(\lambda+a^2)^{\frac{3}{2}}}{\lambda^2}
  \left((a+x)\frac{\lambda+2a^2}{\sqrt{\lambda+a^2}}-
  a\,\frac{\lambda+2(a+x)^2}{\sqrt{\lambda+(a+x)^2}}\right).
\end{equation}
The function $g$ is linear if and only if $a=0$, in which case $g(x)=\omega^2\,x$.

\begin{Remark} If we introduce the new parameters $\xi,\beta$:
\begin{equation}
  \lambda=\frac{1-\xi^2}{\beta},\quad
  a=\frac{\xi}{\sqrt{\beta}},\qquad
  \xi\in(-1,1), \beta>0,
\end{equation}
then \eqref{potentialStillingerinZampForm} takes Stillinger's form (see~\cite{S} formula~(4))
\begin{equation}
  V(x)= \frac{\omega^2}{2(1-\xi^2)^2}
  \left(x+\frac{\xi}{\sqrt{\beta}}
  \biggl(1-\sqrt{1+2\xi\sqrt{\beta}\,x+\beta\,x^2}
  \biggr)\right)^2.
\end{equation}
\end{Remark}

A totally different technique was used in~\cite{z4} to find some explicit examples of local isochronous centers. We are going to show that one of them can be rewritten as formula~\eqref{gStillingerinZampForm} after a suitable reparameterization. Indeed, consider the $g$ defined in formula~(5.12) in~\cite{z4}, that is,
\begin{equation}\label{gCMP}
  g(x)=\frac{2 c\omega^2 }{(b^2 - 4 c)^2}
  \Biggl((b^2+4c)(b + 2 x)  +b\,\frac{b^2-4 c-2
  (b + 2 x)^2}{{{\sqrt{1 +x(b+x)/c}}}}\Biggr).
\end{equation}
Here $\omega>0$, $b,c\in \R$, $c\ne 0$, $b^2-4c\ne 0$. Each of these functions $g$ in~\eqref{gCMP} gives a global center on the whole $\R^2$ whenever it is defined on the whole $\R$, that is, if and only if $1 +x(b+x)/c>0$, so that  its square root in the denominator in~\eqref{gCMP} exists for all $x\in\R$ and never vanishes. This is equivalent to $b^2-4c<0$. By introducing the new parameters $\lambda:=4c-b^2$, $a:=b/2$ we have that the condition
$\lambda>0$ is necessary and sufficient for $g$ to be defined on the whole~$\R$. Notice that this implies $c>0$, in particular $c\ne 0$.  With the new parameters and for $\lambda>0$ we reduce to~\eqref{gStillingerinZampForm}. 

The functions $g$ of formula~\eqref{gCMP} were found by searching for $g$ such that the 4-dimensional system $\ddot x=-g(x)$, $\ddot y=-g'(x)y$, where $g\in C^1$, $g(0)=0$, $g'(0)>0$, becomes superintegrable. More precisely, the system admits a first integral which is positive definite at the origin, so that it constrains the orbits near the origin on compact sets. That method only proves that all these functions $g$ give a \emph{local} isochronous center near the origin for $\ddot x=-g(x)$. The problem of whether the center is indeed \emph{globally} isochronous was not addressed in the original paper, but it can be also recovered by an analyticity argument, without appealing to Theorem~\ref{isochronous}.

\section{Dorignac global potential revisited}\label{dorignac}

Let us consider
\begin{equation}
f(x,y)=\left(\sqrt{1+8e^{3\beta x}}-1\right)\left(\sqrt{1+8e^{3\beta y}}-1\right)-4,\ (x,y)\in\R^2,
\end{equation}
for the real parameters $\omega,\beta\ne 0$. We have $f(0,0)=0$, $f(x,y)=f(y,x)$, and
\begin{equation*}
  \partial_2 f(x,y)=
  12\beta e^{3\beta y}
  \frac{\sqrt{1+8e^{3\beta x}}-1}{\sqrt{1+8e^{3\beta y}}}
  \ne 0,\quad
  (x,y)\in\R^2.
\end{equation*}
The set $f(x,y)=0$ projects itself on the whole $x$-axis since the function $x\mapsto \sqrt{1+8e^{3\beta x}}-1$ is a diffeomorphism of~$\R$ onto $(0,+\infty)$. By Theorem~\ref{globalinvolution}, the equation $f(x,y)=0$ defines a global involution $h:\R\to\R$.  Also now we can solve explicitly the equation and get
\begin{equation}\label{hDorignac1}
  h(x)=\frac{1}{\beta}
  \ln\frac{2\,e^{\beta x}}{\sqrt{1+8e^{3\beta x}}-1},
  \qquad x\in\R.
\end{equation}
Next, formula~\eqref{Visochronous} of Theorem~\ref{isochronous} gives the following isochronous potential which is one of Dorignac's explicit global examples (see formula~(35) in~\cite{dorignac})
\begin{equation}
  V(x)=\frac{\omega^2}{8\beta^2}
  \biggl(\ln\frac{\sqrt{1+8e^{3\beta x}}-1}{2}\biggr)^2,
  \quad x\in\R,\ \omega,\beta\ne 0.
\end{equation}


\end{document}